\documentclass[a4paper,10pt,reqno]{amsart}

\usepackage[utf8]{inputenc}
\usepackage[T1]{fontenc}
\usepackage{amsthm}
\usepackage{amsmath}
\usepackage{amssymb}
\usepackage[inline]{enumitem}
\usepackage{comment}
\PassOptionsToPackage{hyphens}{url}\usepackage{hyperref}
\usepackage{fancyhdr}
\usepackage{mathrsfs}
\usepackage{stmaryrd}
\usepackage[normalem]{ulem}
\usepackage{xcolor}

\setlist[description]{%
  itemsep=0.05cm,               
  font={\normalfont\textsc}, 
 leftmargin=\parindent,
 labelindent=\parindent
}

\theoremstyle{definition}

\newtheorem{theorem}{Theorem}[section]
\newtheorem{lemma}[theorem]{Lemma}
\newtheorem{proposition}[theorem]{Proposition}

\theoremstyle{definition}

\newtheorem*{question*}{Question}

\newtheorem{remarks}[theorem]{Remarks}

\definecolor{blue-url}{RGB}{0,0,100}
\definecolor{red-url}{RGB}{100,0,0}
\definecolor{green-url}{RGB}{0,100,0}
\definecolor{light-yellow}{RGB}{255,255,128}
\definecolor{light-blue}{RGB}{193,255,255}
\definecolor{light-red}{RGB}{239,83,80}

\hypersetup{
	pdftitle={The Automorphism Group of the Finitary Power Monoid of the Additive Group of Integers},
	pdfauthor={Tringali and Wen},
	pdfmenubar=false,
	pdffitwindow=true,
	pdfstartview=FitH,
	colorlinks=true,
	linkcolor=blue-url,
	citecolor=green-url,
	urlcolor=red-url
}

\renewcommand{\,}{\kern 0.1em}

\providecommand\llb{\llbracket}
\providecommand\rrb{\rrbracket}

\providecommand\aut{{\rm Aut}}

\newcommand{\evid}[1]{\textsf{#1}}
\newcommand{\fin}{\mathrm{fin}}

{\newline\vspace{\abovedisplayskip}\hbox to \textwidth\bgroup\hss$\displaystyle}
{$\hss\egroup\vspace{\belowdisplayskip}}

\makeatletter
\DeclareFontFamily{OMX}{MnSymbolE}{}
\DeclareSymbolFont{MnLargeSymbols}{OMX}{MnSymbolE}{m}{n}
\SetSymbolFont{MnLargeSymbols}{bold}{OMX}{MnSymbolE}{b}{n}
\DeclareFontShape{OMX}{MnSymbolE}{m}{n}{
	<-6>  MnSymbolE5
	<6-7>  MnSymbolE6
	<7-8>  MnSymbolE7
	<8-9>  MnSymbolE8
	<9-10> MnSymbolE9
	<10-12> MnSymbolE10
	<12->   MnSymbolE12
}{}
\DeclareFontShape{OMX}{MnSymbolE}{b}{n}{
	<-6>  MnSymbolE-Bold5
	<6-7>  MnSymbolE-Bold6
	<7-8>  MnSymbolE-Bold7
	<8-9>  MnSymbolE-Bold8
	<9-10> MnSymbolE-Bold9
	<10-12> MnSymbolE-Bold10
	<12->   MnSymbolE-Bold12
}{}

\let\llangle\@undefined
\let\rrangle\@undefined
\DeclareMathDelimiter{\llangle}{\mathopen}%
{MnLargeSymbols}{'164}{MnLargeSymbols}{'164}
\DeclareMathDelimiter{\rrangle}{\mathclose}%
{MnLargeSymbols}{'171}{MnLargeSymbols}{'171}
\makeatother

\begin{document}
\title{The Automorphism Group of the Finitary Power Monoid \\ of the Integers under Addition}
\author{Salvatore Tringali}
\address{(S.~Tringali) School of Mathematical Sciences, Hebei Normal University | Shijiazhuang, Hebei province, 050024 China}
\email{salvo.tringali@gmail.com}

\author{Kerou Wen}
\address{(K.~Wen) School of Mathematical Sciences, Hebei Normal University | Shijiazhuang, Hebei province, 050024 China}
\email{kerou.wen.hebnu@outlook.com}

\subjclass[2020]{Primary 08A35, 11P99. Secondary 20M13.}
\keywords{Automorphism group, finitary power monoid, sumset.}
\begin{abstract}

Endowed with the binary operation of set addition carried over from the integers, the family $\mathcal P_{\mathrm{fin}}(\mathbb Z) $ of all non-empty finite subsets of $\mathbb Z$ forms a monoid whose neutral element is the singleton $\{0\}$. 
  
Building upon recent work by Tringali and Yan [J.\ Combin.\ Theory Ser.\ A, 2025], we determine the automorphisms of $\mathcal P_{\mathrm{fin}}(\mathbb Z)$.
In particular, we find that the automorphism group of $\mathcal P_{\mathrm{fin}}(\mathbb Z)$ is isomorphic to the direct product of a cyclic group of order two by the infinite dihedral group.

\end{abstract}

\maketitle

\thispagestyle{empty}

\section{Introduction}
\label{sec:intro}

Let $H$ be a semigroup (see the end of the section for notation and terminology). Endowed with the binary operation of setwise multiplication 
\[
(X, Y) \mapsto XY := \{xy \colon x \in X, \, y \in Y\}
\]
induced by $H$, the family of all non-empty finite subsets of $H$ becomes itself a semigroup, herein denoted by $\mathcal P_\fin(H)$ and named the \evid{finitary power semigroup} of $H$. 

If, in particular, $H$ is a monoid and $1_H$ is its neutral element, then $\mathcal P_\fin(H)$ is also a monoid. Moreover, the finite subsets of $H$ containing $1_H$ form a submonoid of $\mathcal P_\fin(H)$, henceforth denoted by $\mathcal P_{\fin,1}(H)$. Accordingly, we call  $\mathcal P_\fin(H)$ and $\mathcal P_{\fin,1}(H)$ the \evid{finitary power monoid} and the \evid{reduced finitary power monoid} of $H$. Their neutral element is the singleton $\{1_H\}$.

Finitary power semigroups are among the most basic structures in a complex hierarchy generally referred to as \evid{power semigroups}. These objects were first systematically studied by Tamura and Shafer in the late 1960s \cite{Tam-Sha1967}, and more recently, from multiple new perspectives, in a series of papers by Fan and Tringali \cite{Fa-Tr18}, Antoniou and Tringali \cite{An-Tr18}, Bienvenu and Geroldinger \cite{Bie-Ger-22}, Tringali and Yan \cite{Tri-Yan2023(a),Tri-Yan2023(b)}, Garc\'ia-S\'anchez and Tringali \cite{GarSan-Tri2025(a)}, Gonzalez et al.~\cite{Gonz-Li-Rabi-Rodr-Tira-2025}, Aggarwal et al.~\cite{Agga-Gott-Lu-arXiv-21-12}, Dani et al.~\cite{Dani-et-al2025}, Cossu and Tringali \cite{Cos-Tri2025}, etc.

Power semigroups are interesting objects for a number of reasons. First, they provide a most natural algebraic framework for a number of important problems in additive combinatorics and related fields, like S\'{a}rk\"ozy's conjecture \cite[Conjecture 1.6]{Sark2012} on the ``additive irreducibility'' of the set of [non-zero] squares in a finite field of prime order.
Second, the arithmetic of power monoids --- with emphasis on ques\-tions concerned with the possibility or impossibility of factoring certain sets as a product of other sets that are irreducible in a properly defined sense --- is a rich topic in itself and has been pivotal to some radically new developments \cite{An-Tr18,Cos-Tri2025,Cos-Tri2023,Fa-Tr18,Tr20(c)} in factorization theory, a subfield of algebra so far largely confined to commutative domains and cancellative commutative monoids. Third, power semigroups have long been known to play a central role in the theory of automata and formal languages \cite{Alm02, Pin1995}.

In the present paper, motivated by work carried out over the past two years by 
Bienvenu and Geroldinger \cite{Bie-Ger-22} and Tringali and Yan \cite{Tri-Yan2023(a), Tri-Yan2023(b)}, we address the problem of determining the automorphism group of $\mathcal{P}_{\fin}(H)$. (Unless otherwise specified, the word ``morphism'' will always refer to a \textit{semigroup homomorphism}.)
Apart from being an interesting question in its own right (for the automorphisms of an object in a given category are a measure of the ``symmetries'' of the object itself, and their investigation is a fundamental problem in many areas of mathematics), a better understanding of the algebraic properties of power semigroups will hopefully lead to a deeper understanding of other aspects of their theory.

To set the stage, let $\aut(S)$ be the automorphism group of a semigroup $S$, henceforth called the \evid{auto\-group} of $S$. It is not difficult to show (see \cite[Remark 1.1]{Tri-Yan2023(a)}) that, for each $f \in \aut(H)$, the map
$$
\mathcal P_{\fin}(H) \to \mathcal P_{\fin}(H) \colon X \mapsto f[X] := \{f(x) \colon x \in X\}
$$
is an automorphism of $\mathcal P_{\fin}(H)$, hereby referred to as the (\evid{finitary}) \evid{augmentation} of $f$. We say that an auto\-mor\-phism of $\mathcal P_{\fin}(H)$ is \evid{inner} if it is the augmentation of an automorphism of $H$.

Accordingly, we have a well-defined map $\Phi \colon \mathrm{Aut}(H) \to \mathrm{Aut}(\mathcal P_{\fin}(H))$ sending an automorphism of $H$ to its aug\-men\-ta\-tion.
In addition, it is easily checked that
$$
\Phi(f \circ g)(X) = f \circ g[X] = f[g[X]] = f[\Phi(g)(X)] = \Phi(f)(\Phi(g)(X)) = \Phi(f) \circ \Phi(g)(X),
$$
 for all $f, g \in \aut(H)$ and $X \in \mathcal P_\fin(H)$.
If, on the other hand, $f(x) \ne g(x)$ for some $x \in H$ and we take $X := \{x\}$, then $\Phi(f)(X) = \{f(x)\} \ne \{g(x)\} = \Phi(g)(X)$ and hence $\Phi(f) \ne \Phi(g)$.

In other words, $\Phi$ is an injective group homomorphism from $\mathrm{Aut}(H)$ to $\mathrm{Aut}(\mathcal P_\fin(H))$. It is therefore natural to ask if $\Phi$ is also surjective (and therefore an iso\-mor\-phism), which is tantamount to saying that every automorphism of $\mathcal P_{\fin}(H)$ is inner. Unsurprisingly, this is not generally true, and determining the autogroup of $\mathcal P_{\fin}(H)$ is usually much harder than determining $\mathrm{Aut}(H)$, as we are going to demonstrate in the basic case where $H$ is the additive group $\mathbb Z_+ := (\mathbb Z, +)$ of the (ring of) integers. 

More precisely, let us denote by $\mathcal P_{\fin}(\mathbb Z)$ the finitary power monoid of $\mathbb Z_+$. In contrast to what we have done so far with an arbitrary semigroup $H$, we will write $\mathcal P_\fin(\mathbb Z)$ additively, so that the operation on $\mathcal P_\fin(\mathbb Z)$ sends a pair $(X, Y)$ of non-empty finite subsets of $\mathbb Z$ to their sumset 
$$
X + Y := \{x+y \colon x \in X, \, y \in Y\}.
$$
In Theorem \ref{thm:the-automorphisms-of-Pfin(Z)}, we determine the automorphisms of $\mathcal P_{\mathrm{fin}}(\mathbb Z)$. Then, in Theorem \ref{thm:autogroup-of-Pfin(Z)}, we obtain that the autogroup of $\mathcal P_{\mathrm{fin}}(\mathbb Z)$ is isomorphic to the direct product $\mathbb Z_2 \times \mathrm{Dih}_\infty$. Here and throughout, $\mathbb Z_2$ is the group of units $\{\pm 1\}$ of (the multiplicative monoid of) the integers (we regard $\mathbb Z_2$ as the canonical realization of a cyclic group of order $2$); and $\mathrm{Dih}_\infty$ is the \evid{infinite dihedral group} \cite[Example 5.5(1)]{Dum-Foo-2004}, i.e., the external semidirect product $\mathbb Z_2 \ltimes_\alpha \mathbb Z_+$ of $\mathbb Z_2$ by $\mathbb Z_+$ with respect to the unique group homo\-mor\-phism $\alpha \colon \mathbb Z_2 \to \allowbreak \aut(\mathbb Z_+)$ that sends $-1$ to the only non-trivial automorphism of $\mathbb Z_+$, i.e., the inversion map $x \mapsto -x$.

We have striven to make our proofs as self-contained as possible. One exception is Lemma \ref{lem:necessary-condition-on-autos}, which relies on \cite[Theorem 3.2]{Tri-Yan2023(b)} for the explicit description of the automorphisms of the reduced finitary power monoid of the monoid of non-negative integers under addition (it turns out that there are only two of them). We propose directions for future research in Sect.~\ref{sect:future}.

\subsection*{Notation and terminology.} We refer the reader to Grillet \cite{Gril-2001} for general background on semigroups and monoids, and to Robinson \cite{robins2012} for general background on groups. In particular, if $H$ and $K$ are semigroups, then $H \cong K$ means that there is a (semigroup) isomorphism $f$ from $H$ to $K$. Moreover, we denote by $M^\times$ the \evid{group of units} (or \evid{invertible elements}) of a monoid $M$. 

We write $\mathbb N$ for the (set of) non-negative integers, and $\mathbb Z$ for the integers. If $a, b \in \mathbb Z$, we define $\llb a, b \rrb$ to be the \evid{discrete interval} $\{x \in \allowbreak \mathbb Z \colon \allowbreak a \le x \le b\}$. Given $k \in \mathbb N$ and $X \subseteq \mathbb Z$, we set 
\[
kX := \{x_1 + \allowbreak \cdots + x_k: x_1, \ldots, x_k \in X\}
\qquad\text{and}\qquad
{-X} := \{-x: x \in X\};
\]
in particular, $0X = \{0\}$. Lastly, for a function $f$ on $\mathbb Z$, we let $-f$ be the map $\mathbb Z \to \mathbb Z \colon x \mapsto -f(x)$; and for a function $F$ on $\mathcal P_\fin(\mathbb Z)$, we let $-F$ be the map $\mathcal P_\fin(\mathbb Z) \to \mathcal P_\fin(\mathbb Z) \colon X \mapsto -F(X)$.

Further notation and ter\-mi\-nol\-o\-gy, if not explained, are standard or should be clear from the context.

\section{Reduced quotients and totally ordered abelian groups}

Below we prove a series of propositions of a certain generality that will be used in Sect.\ \ref{sec:main-results} to derive \textit{necessary} conditions on the  automorphisms of the finitary power monoid of the additive group of integers. While a more focused approach would have spared certain technicalities, the broader generality of our results may be valuable in addressing analogous problems where the integers are replaced by structures with ``comparable properties'' (see Sect.~\ref{sect:future} for a couple of ideas in this direction). 

We begin with some definitions. Given a \textit{commutative} monoid $M$, we denote by $\simeq_M$ the binary relation on (the underlying set of) $M$ whose graph consists of all pairs $(x,y) \in M \times M$ such that $y \in xM^\times$. It is routine to check that $\simeq_M$ is a congruence on $M$, that is, $\simeq_M$ is an equivalence relation with the additional property that $x_1 \simeq_M y_1$ and $x_2 \simeq_M y_2$ imply $x_1 x_2 \simeq_M y_1 y_2$.
We refer to the factor monoid $M/{\simeq_M}$ as the \evid{reduced quotient} of $M$, and to the set $\{y \in M \colon x \simeq_M y\}$ as the \evid{$\simeq_M$-class} of an element $x \in M$.

\begin{proposition}\label{prop:induced-iso-on-reduced-quotients}
Let $H$ and $K$ be commutative monoids. Every (semigroup) isomorphism $\varphi \colon H \to K$ induces a well-defined isomorphism $\widetilde{\varphi}$ from the reduced quotient $H/{\simeq_H}$ of $H$ to the reduced quotient $K/{\simeq_K}$ of $K$, by sending the $\simeq_H$-class of an element $x \in H$ to the $\simeq_K$-class of $\varphi(x)$.
\end{proposition}

\begin{proof}
Let $\varphi^{-1}(\simeq_K)$ be the binary relation on $H$ whose graph consists of all pairs $(x, y) \in H \times H$ such that $\varphi(x) \simeq_K \varphi(y)$. It is easily checked that $\varphi$ is in fact a \textit{monoid} isomorphism, namely, it maps the neutral element $1_H$ of $H$ to the neutral element $1_K$ of $K$ (see, e.g., the last lines of \cite[Sect.~2]{Tri-2023(c)}). 
Since the functional inverse of $\varphi$ is an isomorphism from $K$ to $H$, it follows that $\varphi^{-1}(K^\times) = H^\times$. Consequently, for all $x, y \in H$, we have that $(x,y) \in \varphi^{-1}(\simeq_K)$ if and only if $\varphi(y) \in \varphi(x) K^\times$, if and only if $y \in x H^\times$, if and only if $x \simeq_H y$. By the proof of \cite[Proposition I.2.6]{Gril-2001}, this guarantees that $\widetilde{\varphi}$ is a well-defined isomorphism from $H/\simeq_H$ to $K/\simeq_K$.
\end{proof}

We recall that a \evid{totally ordered group} is a pair $\mathcal G = (G, \preceq)$ consisting of a group $G$ and a total order $\preceq$ on (the underlying set of) $G$ such that $x \preceq y$ implies $uxv \preceq uyv$ for all $u, v \in G$. 

If $\mathcal G = (G, \preceq)$ is a totally ordered group, then it is easily verified that $\mathcal G_\angle := \{x \in G \colon 1_G \preceq x\}$ is a submonoid of $G$, hereby referred to as the \evid{non-negative cone} of $\mathcal G$. Moreover, for every non-empty finite subset $X$ of $G$, there exists a \textit{unique} element $x_\ast \in X$ such that $x_\ast \preceq y$ for all $y \in X$; we refer to $x_\ast$ as the \evid{$\preceq$-minimum} of $X$. Accordingly, we have a well-defined function $\mu_\mathcal{G} \colon \mathcal P_\fin(G) \to G$ that takes a non-empty finite subset of $G$ to its $\preceq$-minimum. We call $\mu_\mathcal{G}$ the \evid{minimizer} of $\mathcal G$.

\begin{remarks}
\label{rem:minimizer}
\begin{enumerate*}[label=\textup{(\arabic{*})}, mode=unboxed]
\item\label{rem:minimizer(1)} The usual ordering on $\mathbb Z$ turns the additive group of integers into a totally ordered (abelian) group, whose non-negative cone is the set $\mathbb N$ of non-negative integers.
\end{enumerate*}

\begin{enumerate*}[label=\textup{(\arabic{*})}, mode=unboxed, resume]
\item\label{rem:minimizer(2)} Let $\mu$ be the minimizer of a totally ordered group $(G, \preceq)$, and let $X, Y \in \mathcal P_\fin(G)$. If $x_\ast := \mu(X)$ and $y_\ast := \mu(Y)$, then $x_\ast \preceq x$ and $y_\ast \preceq y$ for all $x \in X$ and $y \in Y$, which, by the compatibility of the order with the group operation, yields $1_G \preceq x_\ast^{-1} x$ and $x_\ast y_\ast \preceq xy$. Since $1_G = x_\ast^{-1} x_\ast \in x_\ast^{-1} X$ and $x_\ast y_\ast \in XY$, it follows that 
\[
\mu(X)^{-1} X \in \mathcal P_{\fin,1}(G_\angle)
\qquad\text{and}\qquad
\mu(XY) = x_\ast y_\ast = \mu(X) \mu(Y); 
\]
in particular, the latter condition means that $\mu$ is a homomorphism from $\mathcal P_\fin(G)$ to $G$. Note also that $\mu$ is clearly surjective, because $\{a\} \in \mathcal P_\fin(G)$ and $\mu(\{a\}) = a$ for every $a \in G$. 
\end{enumerate*}
\end{remarks}

The previous remarks provide a ``template'' and the basis for the next results.

\begin{proposition}\label{prop:iso-into-the-reduced-finitary-PM}
Let $\mu$ be the minimizer of a totally ordered group $\mathcal G = (G, \preceq)$. Assume $G$ is abelian and denote by $\mathcal Q(G)$ the reduced quotient of $\mathcal P_\fin(G)$. Let 
$\psi$ be the binary relation from $\mathcal Q(G)$ to the reduced finitary power monoid $\mathcal P_{\fin,1}(\mathcal G_\angle)$ of the non-negative cone $\mathcal G_\angle$ of $\mathcal G$ that, for all $\simeq_{\mathcal P_\fin(G)}$-congruent sets $X, Y \in \mathcal P_\fin(G)$, maps the $\simeq_{\mathcal P_\fin(G)}$-class of $X$ to $\mu(Y)^{-1} Y$. Then $\psi$ is an isomorphism $\mathcal Q(G) \to \mathcal P_{\fin,1}(\mathcal G_\angle)$.
\end{proposition}

\begin{proof}
Write $[X]$ for the $\simeq_{\mathcal P_\fin(G)}$-class of a set $X \in \mathcal P_\fin(G)$, and let $Y \in [X]$. By \cite[Proposition 3.2(iv)]{An-Tr18}, the units of $\mathcal P_\fin(G)$ are precisely the one-element subsets of $G$. Consequently, $Y = aX$ for some $a \in G$. Since $\mathcal G$ is a totally ordered group, it follows that $\mu(Y) = a \mu(X)$, which in turn implies that
\[
\mu(Y)^{-1} Y = \mu(X)^{-1} a^{-1} aX = \mu(X)^{-1} X.
\]
This is sufficient to show that $\psi$ is a function from $\mathcal Q(G)$ to $\mathcal P_{\fin,1}(\mathcal G_\angle)$. 

Now, if $Y \in \mathcal P_{\fin,1}(\mathcal G_\angle)$, then $Y \in \mathcal P_\fin(G)$ and $\mu(Y) = 1_G$, with the result that $\psi([Y]) = Y$. If, on the other hand, $\psi([X]) = \psi([Y])$ for some $X, Y \in \mathcal P_\fin(G)$, then the commutativity of $G$ guarantees that
\[
Y = \allowbreak \mu(Y)\mu(X)^{-1} X \in X \mathcal P_\fin(G)^\times, 
\]
which yields $X \simeq_{\mathcal{P}_\fin(G)} Y$ and hence $[X] = [Y]$. Therefore, $\psi$ is both surjective and injective. 

It remains to see that 
$\psi$ is an isomorphism. 
Let $X, Y \in \mathcal P_\fin(G)$. By Remark \ref{rem:minimizer}\ref{rem:minimizer(2)}, $\mu$ is a hom\-o\-mor\-phism $\mathcal P_\fin(G) \to G$, and hence $\mu(XY)^{-1} = \mu(Y)^{-1} \mu(X)^{-1}$. Again by commutativity, this implies
\[
\psi([X][Y]) = \psi([XY]) = \mu(XY)^{-1} XY = \mu(X)^{-1} X \cdot \mu(Y)^{-1} Y = \psi([X]) \,\psi([Y]).
\]
It follows that $\psi$ is an isomorphism from $\mathcal Q(G)$ to $\mathcal P_{\fin,1}(\mathcal G_\angle)$, as wished.
\end{proof}

\begin{proposition}\label{prop:2.3}
Let $\mathcal H = (H, \le)$ and $\mathcal K = (K, \preceq)$ be abelian totally ordered groups, $\mu_\mathcal{H}$ be the min\-i\-miz\-er of $\mathcal H$, and $\varphi$ be an isomorphism from $\mathcal P_\fin(H)$ to $\mathcal P_\fin(K)$. There then exist an isomorphism $f$ from $\mathcal P_{\fin,1}(\mathcal H_\angle)$ to $\mathcal P_{\fin,1}(\mathcal K_\angle)$ and a homomorphism $\alpha$ from $\mathcal P_\fin(H)$ to $K$ such that
\[
\varphi(X) = \alpha(X) f(\mu_\mathcal{H}(X)^{-1} X), \qquad \text{for all } X \in \mathcal P_\fin(H).
\]
\end{proposition}

\begin{proof}
Denote by $\mathcal Q(H)$ and $\mathcal Q(K)$ the reduced quotients of $\mathcal P_\fin(H)$ and  $\mathcal P_\fin(K)$, respectively.
By Prop\-o\-si\-tion \ref{prop:iso-into-the-reduced-finitary-PM}, mapping the $\simeq_{\mathcal P_\fin(H)}$-class of a set $X \in \mathcal P_\fin(H)$ to $\mu_{\mathcal{H}}(X)^{-1} X$ yields a (well-defined) iso\-mor\-phism $\psi_H \colon \mathcal Q(H) \to \mathcal P_{\fin,1}(\mathcal H_\angle)$. In a similar way, mapping the $\simeq_{\mathcal P_\fin(K)}$-class of a set $Y \in \mathcal P_\fin(K)$ to $\mu_{\mathcal{K}}(Y)^{-1} Y$ results in an isomorphism $\psi_K \colon \mathcal Q(K) \to \mathcal P_{\fin,1}(\mathcal K_\angle)$, where $\mu_\mathcal{K}$ is the minimizer of $\mathcal K$. 

In addition, we have from Proposition \ref{prop:induced-iso-on-reduced-quotients} that $\varphi$ induces an isomorphism $\widetilde{\varphi} \colon \mathcal Q(H) \to \mathcal Q(K)$ by sending the $\simeq_{\mathcal P_\fin(H)}$-class of a set $X \in \mathcal P_\fin(H)$ to the $\simeq_{\mathcal P_\fin(K)}$-class of $\varphi(X)$. It follows that the function
\[
f := \psi_K \circ \widetilde{\varphi} \circ \psi_H^{-1}
\]
is an isomorphism $\mathcal P_{\fin,1}(\mathcal H_\angle) \to \mathcal P_{\fin,1}(\mathcal K_\angle)$. More explicitly, if $X \in \mathcal P_{\fin,1}(\mathcal H_\angle)$, then $
\psi_H^{-1}(X) = \{hX \colon h \in H\}$
and hence $\widetilde{\varphi} \circ \psi_H^{-1}(X) = \{k \varphi(X) \colon k \in K\}$. Accordingly, it holds that
\[
f(X) = \psi_K(\widetilde{\varphi} \circ \psi_H^{-1}(X)) = \mu_\mathcal{K}( \varphi(X))^{-1} \varphi(X), \qquad \text{for every } X \in \mathcal P_{\fin,1}(\mathcal H_\angle).
\]
Since $\mu_\mathcal{H}(X)^{-1} X \in \mathcal P_{\fin,1}(\mathcal H_\angle)$ for each $X \in \mathcal P_\fin(H)$, we can thus conclude that, for all $X \in \mathcal P_\fin(H)$,
\begin{equation*}
\begin{split}
\varphi(X) & = \varphi(\mu_\mathcal{H}(X)^{-1} X) \,\varphi(\{\mu_\mathcal{H}(X)\}) = \mu_\mathcal{K}(\varphi(\mu_\mathcal{H}(X)^{-1} X)) f(\mu_\mathcal{H}(X)^{-1} X) \,\varphi(\{\mu_\mathcal{H}(X)\}).
\end{split}
\end{equation*}
This is enough to finish the proof, upon observing that, by Remark \ref{rem:minimizer}\ref{rem:minimizer(2)}, the function 
\[
\mathcal P_\fin(H) \to K' \colon X \mapsto \mu_\mathcal{K}(\varphi(\mu_\mathcal{H}(X)^{-1} X)) \, \varphi(\{\mu_\mathcal{H}(X)\})
\]
is a homomorphism, where $K'$ is the image of $K$ under the embedding $K \to \mathcal P_\fin(K) \colon y \mapsto \{y\}$.
\end{proof}

\section{Main results}\label{sec:main-results}

Throughout, $\mathcal P_{\fin}(\mathbb Z)$ denotes the finitary power monoid of $\mathbb Z_+$ (the additive group of integers), and $\mathcal P_{\fin,0}(\mathbb N)$ denotes the reduced finitary power monoid of the monoid of non-negative integers under addition. We will write each of these monoids additively. We start with a series of lemmas.

\begin{lemma}\label{lem_3.1}
If $f$ is a homomorphism from $\mathcal P_\fin(\mathbb Z)$ to $\mathbb Z_+$, then there exist $a, b \in \mathbb Z$ such that 
\[
f(X) = a \min X + b \max X,
\qquad\text{for all }
X \in \mathcal P_\fin(\mathbb Z).
\]
\end{lemma}

\begin{proof}
Set $a := -f(\{-1,0\})$ and $b := f(\{0,1\})$, and fix $X \in \mathcal P_\fin(\mathbb Z)$. Given $m, n \in \mathbb N$ such that $\max X - \min X \le m + n$ (and hence $- m + \max X \le n + \min X$), we have 
\begin{equation}\label{lem_3.1:eq_(1)}
\begin{split}
\llb -m + \min X, n + \max X \rrb & = \llb -m+\min X, n + \min X \rrb \cup \llb -m + \max X, n + \max X \rrb \\
& = \llb -m, n \rrb + \{\min X, \max X\} \subseteq \llb -m, n \rrb + X \\
& \subseteq \llb -m, n \rrb + \llb \min X, \max X \rrb = \llb -m + \min X, n + \max X \rrb.
\end{split}
\end{equation}
Considering that
$\llb -h, k \rrb = h \{0, -1\} + k \{0, 1\}$ for all $h, k \in \mathbb N$,
and assuming for the remainder that $m \ge \allowbreak \min X$ and $n \ge - \max X$,
it is then immediate from Eq.~\eqref{lem_3.1:eq_(1)} that
\[
\begin{split}
 (m-\min X) \,\{-1,0\} & + (n+\max X)\, \{0,1\} = \llb -m + \min X, n + \max X \rrb \\
 & = \llb -m, n \rrb + X  = m\{-1,0\} + n\{0,1\} + X.
\end{split}
\]
Since $f$ is a homomorphism from $\mathcal P_\fin(\mathbb Z)$ to the additive group of integers, it follows that
\[
-(m - \min X)\, a + (n + \max X)\, b = -ma + nb + f(X),
\]
which, by cancellativity in $\mathbb Z_+$, implies $f(X) = a \min X + b \max X$ and finishes the proof.
\end{proof}

\begin{lemma}
\label{lem:necessary-condition-on-autos}
Every automorphism of $\mathcal P_\fin(\mathbb Z)$ is of the form $\pm \varphi_{\alpha,\beta}$ for some $\alpha, \beta \in \mathbb Z$, where $\varphi_{\alpha,\beta}$ is the function on $\mathcal P_\fin(\mathbb Z)$ defined by $X \mapsto X + \alpha \min X + \beta \max X$.
\end{lemma}

\begin{proof}
Let $\varphi$ be an automorphism of $\mathcal P_\fin(\mathbb Z)$. By Remark \ref{rem:minimizer}\ref{rem:minimizer(1)}, Proposition \ref{prop:2.3} (applied with $H = K = \mathbb Z_+$), and Lemma \ref{lem_3.1}, there exist $f \in \aut(\mathcal P_{\fin,0}(\mathbb N))$ and $a, b \in \mathbb Z$ such that
\begin{equation}\label{equ:lem:necessary-condition-on-autos}
\varphi(X) = f(X - \min X) + a \min X + b \max X, 
\qquad \text{for every } X \in \mathcal P_\fin(\mathbb Z).
\end{equation}
On the other hand, we are guaranteed by \cite[Theorem 3.2]{Tri-Yan2023(b)} that the only automorphisms of $\mathcal P_{\fin,0}(\mathbb N)$ are the identity $\mathrm{id} \colon X \mapsto X$ and the map $\mathrm{rev} \colon X \mapsto \max X-X$. Taking $f = \mathrm{id}$ in Eq.~\eqref{equ:lem:necessary-condition-on-autos} yields
\[
\varphi(X) =  
X + (a-1) \min X + b \max X, 
\qquad \text{for every } X \in \mathcal P_\fin(\mathbb Z);
\]
while taking $f = \mathrm{rev}$ and noting that $\mathrm{rev}(Y + k) = \mathrm{rev}(Y)$ for all $Y \in \mathcal P_\fin(\mathbb Z)$ and $k \in \mathbb Z$ lead to
\[
\varphi(X) = 
{-X} + a \min X + (b+1) \max X,
\qquad \text{for every } X \in \mathcal P_\fin(\mathbb Z).
\]
In the notation of the statement, this means that either $\varphi = \varphi_{a-1,b}$ or $\varphi = {-\varphi_{-a,-(b+1)}}$ (as wished).
\end{proof}

\begin{lemma}
\label{lem:compositions}
Using the notation of Lemma \ref{lem:necessary-condition-on-autos}, define $f_\alpha := \varphi_{\alpha,-\alpha}$ and $g_\alpha := \varphi_{\alpha-1,-(\alpha+1)}$, where $\alpha$ is an arbitrary integer. For all $a, b \in \mathbb Z$, we have 
\begin{equation*}
f_a \circ f_b = f_{a+b} = g_a \circ g_{-b}
\qquad\text{and}\qquad
f_a \circ g_b = g_{a+b} = g_b \circ f_{-a}.
\end{equation*}
\end{lemma}

\begin{proof}
Given $Y \in \mathcal P_\fin(\mathbb Z)$, set $\delta^\pm(Y) := \min Y \pm \max Y$. For every $k \in \mathbb Z$, we have that
\begin{equation}
\label{lem:compositions:eq(1)}
\delta^\pm(Y+k) = \min(Y + k) \pm \max(Y + k) = 
\delta^\pm(Y) + k \pm k.
\end{equation}
It is then clear from our definitions that, for all $\alpha \in \mathbb Z$ and $ X \in \mathcal P_\fin(\mathbb Z)$,
\begin{equation}
\label{lem:compositions:eq(2)}
f_\alpha(X) = X + \alpha\min X - \alpha \max X = X + \alpha\, \delta^-(X)
\end{equation}
and
\begin{equation}
\label{lem:compositions:eq(3)}
g_\alpha(X) = X + (\alpha-1) \min X - (\alpha+1) \max X = X + \alpha\, \delta^-(X) - \delta^+(X)=f_\alpha(X)-\delta^+(X).
\end{equation}
By straightforward calculations, it follows that
\begin{equation}
\label{lem:compositions:eq(4)}
\delta^-(g_\alpha(X)) 
\stackrel{\eqref{lem:compositions:eq(3)}}{=} 
\delta^-(f_\alpha(X)-\delta^+(X))
\stackrel{\eqref{lem:compositions:eq(1)}}{=} 
\delta^-(f_\alpha(X))
\stackrel{\eqref{lem:compositions:eq(2)}}{=} 
\delta^-(X+\alpha\,\delta^-(X))
\stackrel{\eqref{lem:compositions:eq(1)}}{=} 
\delta^-(X),
\end{equation}
and in a similar way,
\begin{equation}
\label{lem:compositions:eq(8)}
\delta^+(g_\alpha(X)) 
\stackrel{\eqref{lem:compositions:eq(3)}}{=} \delta^+(X + \alpha \, \delta^-(X) - \delta^+(X)) \stackrel{\eqref{lem:compositions:eq(1)}}{=} 
2\alpha \, \delta^-(X) - \delta^+(X),
\end{equation}
where the notation $\stackrel{(\ast)}{=}$ means that the equality is justified by Eq.~($\ast$).

Now, fix $a, b \in \mathbb Z$ and $X \in \mathcal P_\fin(\mathbb Z)$. From the previous equations, we readily see that
\begin{equation}
\label{lem:compositions:eq(5)}
\begin{split}
f_a \circ f_b(X) 
\stackrel{\eqref{lem:compositions:eq(2)}}{=}
f_b(X) + a\, \delta^-(f_b(X)) 
\stackrel{\eqref{lem:compositions:eq(4)}}{=}
f_b(X) + a\, \delta^-(X) 
\stackrel{\eqref{lem:compositions:eq(2)}}{=}
f_{a+b}(X)
\end{split}
\end{equation}
and 
\begin{equation}
\label{lem:compositions:eq(6)}
\begin{split}
f_a \circ g_b(X) 
\stackrel{\eqref{lem:compositions:eq(2)}}{=}
g_b(X) + a\,\delta^-(g_b(X)) 
\stackrel{\eqref{lem:compositions:eq(4)}}{=}
g_b(X) + a\, \delta^-(X) 
\stackrel{\eqref{lem:compositions:eq(3)}}{=}
g_{a + b}(X).
\end{split}
\end{equation}
Consequently, we obtain 
\begin{equation}
\label{lem:compositions:eq(10)}
\begin{split}
g_a \circ g_{-b}(X)
& 
\stackrel{\eqref{lem:compositions:eq(3)}}{=}
f_a(g_{-b}(X)) - \delta^+(g_{-b}(X)) 
\stackrel{\eqref{lem:compositions:eq(6)}}{=}
g_{a-b}(X) - \delta^+(g_{-b}(X)) \\
&
\stackrel{\eqref{lem:compositions:eq(8)}}{=}
g_{a-b}(X) + 2b\, \delta^-(X) + \delta^+(X)
\stackrel{\eqref{lem:compositions:eq(3)}}{=}
f_{a-b}(X) + 2b\, \delta^-(X) 
\stackrel{\eqref{lem:compositions:eq(2)}}{=}
f_{a+b}(X)
\end{split}
\end{equation}
and
\begin{equation}
\label{lem:compositions:eq(11)}
\begin{split}
g_b \circ f_{-a}(X) 
& 
\stackrel{\eqref{lem:compositions:eq(2)}}{=}
g_b(X - a\, \delta^{-}(X)) 
\stackrel{\eqref{lem:compositions:eq(3)}}{=}
(X-a\delta^-(X))+b\,\delta^-(X-a\,\delta^-(X))-\delta^+(X-a\,\delta^-(X)) \\
&\stackrel{\eqref{lem:compositions:eq(1)}}{=}
(X-a\delta^-(X)) + b\,\delta^-(X) - \delta^+(X)+2a\,\delta^-(X)
\stackrel{\eqref{lem:compositions:eq(3)}}{=}
g_{a+b}(X). 
\end{split}
\end{equation}
Since $X$ is an arbitrary set in $\mathcal P_\fin(\mathbb Z)$, Eqs.~\eqref{lem:compositions:eq(5)}--\eqref{lem:compositions:eq(11)} lead to the desired conclusion.
\end{proof}

We are now in a position to prove the first of our main results.

\begin{theorem}
\label{thm:the-automorphisms-of-Pfin(Z)}
The automorphisms of $\mathcal P_\fin(\mathbb Z)$ are precisely the endofunctions of $\mathcal P_\fin(\mathbb Z)$ of the form $\pm f_\alpha$ and $\pm g_\alpha$, where $\alpha$ is an arbitrary integer and, for all $X \in \mathcal P_\fin(\mathbb Z)$, we define 
\[
f_\alpha(X) := X + \alpha \min X - \alpha \max X
\qquad\text{and}\qquad
g_\alpha(X) := X + (\alpha-1) \min X - (\alpha+1) \max X.
\]
\end{theorem}

\begin{proof}
Let $\varphi \in \aut(\mathcal P_\fin(\mathbb Z))$. It is immediate to verify that $-\varphi$ is also an automorphism of $\mathcal P_\fin(\mathbb Z)$. 
By Lemma \ref{lem:necessary-condition-on-autos}, we may therefore assume without loss of generality that there exist $a, b \in \mathbb Z$ 
such that 
\begin{equation}\label{equ:(4)}
\varphi(X) = X + a \min X + b \max X, 
\qquad\text{for all } X \in \mathcal P_\fin(\mathbb Z).
\end{equation}
We seek necessary conditions on the pair $(a,b)$ for $\varphi$ to be bijective. To start with, Eq.~\eqref{equ:(4)} yields 
\begin{equation}
\label{equ:system-of-equations}
\begin{cases}
\min \varphi(X) = (a+1) \min X + b \max X, \\
\max \varphi(X) = a \min X + (b+1)\max X,
\end{cases}
\qquad\text{for every } X \in \mathcal P_\fin(\mathbb Z).
\end{equation}
Fix $c, d \in \mathbb{Z}$ with $c \le d$. By the surjectivity of $\varphi$, there exists $X \in \mathcal{P}_{\fin}(\mathbb{Z})$ such that $\min \varphi(X) = c$ and $\max \varphi(X) = d$. It follows by Eq.~\eqref{equ:system-of-equations} that the linear system in the variables $y$ and $z$ given by
\[
\begin{cases}
(a+1) y + b z = c, \\
a  y + (b + 1) z = d,
\end{cases}
\]  
has at least one solution with $y, z \in \mathbb Z$. Subtracting the second equation of the system from the first leads to $y = z + c - d$, which, when substituted back into the second equation, results in
\begin{equation*}
(a+b+1)z = d + ad - ac.
\end{equation*}

This ultimately proves that, for $\varphi$ to be surjective, $a+b+1$ must divide $d+ad-ac$ for all $c, d \in \mathbb Z$ with $c \le d$. In particular, taking $c = d = 1$ implies that $a+b+1 \in \{\pm 1\}$, and hence 
\[
a+b = 0 \qquad\text{or}\qquad a+b = -2.
\]
Therefore, in the notation of the statement, we find that either $\varphi = f_\alpha$ or $\varphi = g_\alpha$ for some $\alpha \in \mathbb Z$. We are left with showing that any such function is indeed an automorphism of $\mathcal P_\fin(\mathbb Z)$.

Let $\alpha \in \mathbb Z$. Since the maps $X \mapsto \min X$ and $X \mapsto \max X$ are both homomorphisms from $\mathcal P_\fin(\mathbb Z)$ to $(\mathbb Z, +)$, it is clear that each of $f_\alpha$ and $g_\alpha$ is an endomorphism of $\mathcal P_\fin(\mathbb Z)$. On the other hand, we know from Lemma \ref{lem:compositions} that $f_{\pm \alpha} \circ f_{\mp \alpha} = g_\alpha \circ g_\alpha = f_0$. Recognizing that $f_0$ is the identity function on $\mathcal P_\fin(\mathbb Z)$, we can thus conclude that $f_\alpha$ and $g_\alpha$ are bijections, thereby completing the proof.
\end{proof}

We finish by characterizing the autogroup of $\mathcal P_\fin(\mathbb Z)$. To this end, recall from Sect.~\ref{sec:intro} that we denote by $\mathbb Z_2$ the group of units of the ring of integers and by $\text{Dih}_\infty$ the infinite dihedral group.

\begin{theorem}\label{thm:autogroup-of-Pfin(Z)}
$\aut(\mathcal P_\fin(\mathbb Z)) \cong \mathbb Z_2 \times \text{Dih}_\infty$.
\end{theorem}

\begin{proof}
Given an integer $\alpha \in \mathbb{Z}$, let $f_\alpha$ and $g_\alpha$ be defined as in Lemma \ref{lem:compositions}. Accordingly, set 
\[
E := \{f_0, g_0\},
\qquad F := \{f_\alpha: \alpha \in \mathbb Z\}, \qquad
G := \{g_\alpha: \alpha \in \mathbb Z\},
\qquad\text{and}\qquad
H := F \cup G.
\]
It is immediate from Lemma \ref{lem:compositions} that each of $E$, $F$, and $H$ is a subgroup of $\aut(\mathcal P_\fin(\mathbb Z))$; in particular, note that $f_0$ is the identity automorphism of $\mathcal P_\fin(\mathbb Z)$, and observe that  $f_\alpha^{-1} = f_{-\alpha}$ and $g_\alpha^{-1} = g_\alpha$ for every $\alpha \in \mathbb Z$. Furthermore, $F$ is a normal subgroup of $H$, as we check that
\begin{equation}\label{thm_3.5:eq_(0)}
\phi \circ f_\alpha \circ \phi^{-1} = 
\begin{cases}  
f_\alpha & \text{if } \phi \in F, \\  
f_{-\alpha} & \text{if } \phi \in G,
\end{cases}  
\qquad \text{for every } \phi \in H \text{ and } \alpha \in \mathbb Z.
\end{equation}
Considering that $E \cap F = \{f_0\}$ and $EF = F \cup g_0 F = F \cup G = H$, this ultimately shows that $H = E \ltimes F$, namely, $H$ is the internal semidirect product of $E$ by $F$. We claim that 
\begin{equation}\label{thm_3.5:eq_(1)}
H = E \ltimes F \cong \text{Dih}_\infty. 
\end{equation}

Indeed, $F$ is an infinite cyclic group generated by $f_1$, as we have from Lemma \ref{lem:compositions} (and an elementary induction) that $f_\alpha$ is the $\alpha$-fold composition of $f_1$ for all $\alpha \in \mathbb{Z}$. On the other hand, $E$ is a cyclic group of order $2$ and, by Eq.~\eqref{thm_3.5:eq_(0)},
 the action of $g_0$ on $F$ by conjugation satisfies $
g_0 \circ f_\alpha \circ g_0^{-1} = f_{-\alpha} = \allowbreak f_\alpha^{-1}$ for each $
\alpha \in \mathbb Z$. Consequently, we gather from \cite[p.~51, Example (I)]{robins2012} that $E \ltimes F \cong \text{Dih}_\infty$ (as wished).  

Now, Theorem \ref{thm:the-automorphisms-of-Pfin(Z)} guarantees that $\aut(\mathcal P_\fin(\mathbb Z)) = \{\pm \phi: \phi \in H\}$. It is then routine to check (cf.~the first line of the proof of Theorem \ref{thm:the-automorphisms-of-Pfin(Z)}) that the function 
$(\pm 1, \phi) \mapsto \pm \phi$
provides an isomorphism from $\mathbb Z_2 \times H$ to $\aut(\mathcal P_\fin(\mathbb Z))$. Since it is a basic fact that $S_1 \times T_1 \cong S_2 \times T_2$ for all semigroups $S_1$, $S_2$, $T_1$, and $T_2$ such that $S_1 \cong S_2$ and $T_1 \cong T_2$, we can therefore conclude from Eq.~\eqref{thm_3.5:eq_(1)} that $\aut(\mathcal P_\fin(\mathbb Z)) \cong \mathbb Z_2 \times \allowbreak H \cong \mathbb Z_2 \times \text{Dih}_\infty$. 
\end{proof}

\section{Prospects for future research}
\label{sect:future}

Let $(G, \preceq)$ be a totally ordered group (as defined in the comment just before Proposition \ref{prop:iso-into-the-reduced-finitary-PM}), and recall from Sect.~\ref{sec:intro} that the reduced finitary power monoid $\mathcal P_{\fin,1}(G)$ of $G$ is the submonoid of $\mathcal P_\fin(G)$ consisting of all finite subsets of $G$ that contain the neutral element $1_G$. 

\begin{question*}
Is it true that every automorphism $\varphi$ of the \textit{reduced} finitary power monoid $\mathcal P_{\fin,1}(G)$ of $G$ is \evid{inner}, meaning that there exists an automorphism $f$ of $H$ such that $\varphi(X) = f[X]$ for every non-empty finite set $X \subseteq G$ containing the neutral element $1_G$? In particular, does this hold when $G$ is abelian?
\end{question*}

We conjecture a positive answer when $G = \mathbb Z_+$ (the additive group of integers), which, by Theorem \ref{thm:autogroup-of-Pfin(Z)}, amounts to proving that the only non-trivial automorphism of the reduced finitary power monoid of $\mathbb Z_+$ is given by $X \mapsto -X$ (recall that the only non-trivial automorphism of $\mathbb Z_+$ is the map $x \mapsto -x$). 

\section*{Acknowledgments}

The authors were both supported by grant no.~A2023205045, funded by the Natural Science Foun\-da\-tion of Hebei Province. They are grateful to Jun Seok Oh (Jeju National University, South Korea) for the reference to Grillet's book \cite{Gril-2001} used in the proof of Proposition \ref{prop:induced-iso-on-reduced-quotients}, and to Pedro Garc\'ia-S\'anchez (University of Granada, Spain) for many valuable suggestions that improved the presentation.

\end{document}